\documentclass[10pt]{article}
\usepackage{amssymb}
\usepackage{amsthm}
\usepackage{latexsym}
\usepackage{comment}
\usepackage{hyperref}
\usepackage[all]{xy}
\usepackage{graphicx}

\title{Images of $2$-adic representations associated to hyperelliptic Jacobians}
\author{Jeffrey Yelton}

\newtheorem{thm}{Theorem}[section]
\newtheorem{prop}[thm]{Proposition}

\newtheorem{cor}[thm]{Corollary}

\theoremstyle{definition} \newtheorem{rmk}[thm]{Remark}
\newcommand{\cc}{\mathbb{C}}

\newcommand{\qq}{\mathbb{Q}}
\newcommand{\zz}{\mathbb{Z}}

\newcommand{\Gal}{\mathrm{Gal}}

\begin{document}

\maketitle

\begin{abstract}

Let $k$ be a subfield of $\cc$ which contains all $2$-power roots of unity, and let $K = k(\alpha_{1}, \alpha_{2}, ... , \alpha_{2g + 1})$, where the $\alpha_{i}$'s are independent and transcendental over $k$, and $g$ is a positive integer.  We investigate the image of the $2$-adic Galois action associated to the Jacobian $J$ of the hyperelliptic curve over $K$ given by $y^{2} = \prod_{i = 1}^{2g + 1} (x - \alpha_{i})$.  Our main result states that the image of Galois in $\mathrm{Sp}(T_{2}(J))$ coincides with the principal congruence subgroup $\Gamma(2) \lhd \mathrm{Sp}(T_{2}(J))$.  As an application, we find generators for the algebraic extension $K(J[4]) / K$ generated by coordinates of the $4$-torsion points of $J$.

\end{abstract}

\section{Introduction}

Fix a positive integer $g$.  An affine model for a hyperelliptic curve over $\cc$ of genus $g$ may be given by 
\begin{equation}\label{hyperelliptic model} y^{2} = \prod_{i = 1}^{2g + 1} (x - \alpha_{i}), \end{equation}
with $\alpha_{i}$'s distinct complex numbers.  Now let $\alpha_{1}, ... , \alpha_{2g + 1}$ be transcendental and independent over $\cc$, and let $L$ be the subfield of $\cc(\alpha) := \cc(\alpha_{1}, ... , \alpha_{2g + 1})$ generated over $\cc$ by the elementary symmetric functions of the $\alpha_{i}$'s.  For any positive integer $N$, let $J[N]$ denote the $N$-torsion subgroup of $J(\bar{L})$.  For each $n \geq 0$, let $L_{n} = L(J[2^{n}])$ denote the extension of $L$ over which the $2^{n}$-torsion of $J$ is defined.  Set
$$L_{\infty} := \bigcup_{n = 1}^{\infty} L_{n}.$$
Note that $\cc(\alpha_{1}, ... \alpha_{2g + 1})$ is Galois over $L$ with Galois group isomorphic to $S_{2g + 1}$.  It is well known (\cite{Mumford1984tata}, Corollary 2.11) that $\cc(\alpha_{1}, ... , \alpha_{2g + 1}) = L_{1}$, so $\mathrm{Gal}(L_{1} / L) \cong S_{2g + 1}$.  Fix an algebraic closure $\bar{L}$ of $L$, and write $G_{L}$ for the absolute Galois group $\mathrm{Gal}(\bar{L} / L)$.

Let $C$ be the curve defined over $L$ by equation (\ref{hyperelliptic model}), and let $J / L$ be its Jacobian.  For any prime $\ell$, let   
$$T_{\ell}(J) := \lim_{\leftarrow n} J[\ell^{n}]$$
denote the $\ell$-adic Tate module of $J$; it is a free $\zz_{\ell}$-module of rank $2g$ (see \cite{mumford1974abelian}, \S18).  For the rest of this paper, we write $\rho_{\ell} : G_{L} \to \mathrm{Aut}(T_{\ell}(J))$ for the continuous homomorphism induced by the natural Galois action on $T_{\ell}(J)$.  Write $\mathrm{SL}(T_{\ell}(J))$ (resp. $\mathrm{Sp}(T_{\ell}(J))$) for the subgroup of automorphisms of the $2$-adic Tate module $T_{\ell}(J)$ with determinant $1$ (resp. automorphisms of $T_{\ell}(J)$ which preserve the Weil pairing).  Since $L$ contains all $2$-power roots of unity, the Weil pairing on $T_{2}(J)$ is Galois invariant, and it follows that the image of $\rho_{2}$ is contained in $\mathrm{Sp}(T_{2}(J))$.  For each $n \geq 0$, we denote by 
$$\Gamma(2^{n}) := \{g \in \mathrm{Sp}(T_{2}(J))\ |\ g \equiv 1\ \mathrm{(mod}\ 2^{n}\mathrm{)}\} \lhd \mathrm{Sp}(T_{2}(J))$$
the level-$2^{n}$ principal congruence subgroup of $\mathrm{Sp}(T_{2}(J))$.

Our main theorem is the following.

\begin{thm}\label{thm: section 2 main theorem}

With the above notation, the image under $\rho_{2}$ of the Galois subgroup fixing $L_{1}$ is $\Gamma(2) \lhd \mathrm{Sp}(T_{2}(J))$.

\end{thm}

Before setting out to prove this theorem, we state some easy corollaries.

\begin{cor}\label{prop: description of G}

Let $G$ denote the image under $\rho_{2}$ of all of $G_{L}$.  Then we have the following: 

a) $G$ contains $\Gamma(2) \lhd \mathrm{Sp}(T_{2}(J))$, and $G / \Gamma(2) \cong S_{2g + 1}$.

b) In the case that $g = 1$, $G = \mathrm{Sp}(T_{2}(J)) = \mathrm{SL}(T_{2}(J))$.

c) For each $n \geq 1$, the homomorphism $\rho_{2}$ induces an isomorphism 
$$\bar{\rho}_{2}^{(n)} : \mathrm{Gal}(L_{n} / L_{1}) \stackrel{\sim}{\to} \Gamma(2) / \Gamma(2^{n})$$
 via the restriction map $\mathrm{Gal}(\bar{L} / L_{1}) \twoheadrightarrow \mathrm{Gal}(L_{n} / L_{1})$.

\end{cor}

\begin{proof}

Since $\mathrm{Gal}(L_{1} / L) \cong S_{2g + 1}$, part (a) immediately follows from the theorem.  If $g = 1$, then fix a basis of $T_{2}(J)$ so that we may identify $\mathrm{Sp}(T_{2}(J))$ (resp. $\mathrm{SL}(T_{2}(J))$) with $\mathrm{Sp}_{2}(\zz_{2})$ (resp. $\mathrm{SL}_{2}(\zz_{2})$).  Then it is well known that $\mathrm{Sp}_{2}(\zz_{2}) = \mathrm{SL}_{2}(\zz_{2})$, and that $\mathrm{SL}_{2}(\zz_{2}) / \Gamma(2) \cong \mathrm{SL}_{2}(\zz / 2\zz) \cong S_{3}$.  Since, by part (a), $G / \Gamma(2) \cong S_{3}$ when $g = 1$, the linear subgroup $G$ must be all of $\mathrm{Sp}(T_{2}(J)) = \mathrm{SL}(T_{2}(J))$, which is the statement of (b).  To prove part (c), note that for any $n \geq 0$, the image under $\rho_{2}$ of the Galois subgroup fixing the $2^{n}$-torsion points is clearly $G \cap \Gamma(2^{n})$.  But $G > \Gamma(2)$, so for any $n \geq 1$, the image under $\rho_{2}$ of $\mathrm{Gal}(\bar{L} / L(2^{n}))$ is $\Gamma(2^{n})$.  Then part (c) immediately follows by the definition of $\bar{\rho}_{2}^{(n)}$.

\end{proof}

In \S2, we will prove the main theorem by considering a family of hyperelliptic curves whose generic fiber is $C$.  In \S3, we will use the results of the previous two sections to determine generators for the algebraic extension $L_{2} / L$ (Theorem \ref{prop: 2 and 4 case}).  Finally, in \S4, we will generalize Theorems \ref{thm: section 2 main theorem} and \ref{prop: 2 and 4 case} by descending from $\cc$ to a subfield $k \subset \cc$ which contains all $2$-power roots of unity.

\section{Families of hyperelliptic Jacobians}

In order to prove Theorem \ref{thm: section 2 main theorem}, we study a family of hyperelliptic curves parametrized by all (unordered) $(2g + 1)$-element subsets $T = \{\alpha_{i}\} \subset \cc$ whose generic fiber is $C$.  Let $e_{1} := \sum_{i = 1}^{2g + 1} \alpha_{i}, ... , e_{2g + 1} := \prod_{i = 1}^{2g + 1} \alpha_{i}$ be the elementary symmetric functions of the variables $\alpha_{i}$, and let $\Delta$ be the discriminant function of these variables.  Then the base of this family is the affine variety over $\cc$ given by 
\begin{equation}\label{configuration space} X := \mathrm{Spec}(\cc[e_{1}, e_{2}, ... , e_{2g + 1}, \Delta^{-1}]). \end{equation}
This complex affine scheme may be viewed as the configuration space of $(2g + 1)$-element subsets of $\cc$ (see the discussion in Section 6 of \cite{Yu1997toward}).  More precisely, we identify each $\cc$-point $T = (e_{1}, e_{2}, ... , e_{2g + 1})$ of $X$ with the set of roots of the squarefree degree-$(2g + 1)$ polynomial $z^{2g + 1} - e_{1}z^{2g} + e_{2}z^{2g - 1} - ... - e_{2g + 1} \in \cc[z]$, which is a $(2g + 1)$-element subset of $\cc$.  Note that the function field of $X$ is $L$.  The (topological) fundamental group of $X$ is isomorphic to $B_{2g + 1}$, the braid group on $2g + 1$ strands.  The braid group $B_{2g + 1}$ is generated by elements $\sigma_{1}$, $\sigma_{2}$, ... , $\sigma_{2g}$, with relations $\sigma_{1}\sigma_{i + 1}\sigma_{i} = \sigma_{i + 1}\sigma_{i}\sigma_{i + 1}$ for $1 \leq i \leq 2g$ and $\sigma_{i}\sigma_{j} = \sigma_{j}\sigma_{i}$ for $2 \leq i + 1 < j \leq 2g$.  (See section 1.4 of \cite{birman1974braids} for more details.)

We also define the complex affine scheme 
\begin{equation}\label{ordered configuration space} Y := \mathrm{Spec}(\cc[\alpha_{1}, \alpha_{2}, ... , \alpha_{2g + 1}, \{(\alpha_{i} - \alpha_{j})^{-1}\}_{1 \leq i < j \leq 2g + 1}]). \end{equation}
As a complex manifold, $Y$ is the \textit{ordered} configuration space, whose $\cc$-points may be identified with $2g + 1$-element subsets of $\cc$ which are given an ordering (a $\cc$-point is identified with its coordinates $(\alpha_{1}, \alpha_{2}, ... , \alpha_{2g + 1})$).  There is an obvious covering map $Y \to X$ which sends each point $(\alpha_{1}, \alpha_{2}, ... , \alpha_{2g + 1})$ of $Y$ to the point in $X$ corresponding to the (unordered) subset $\{\alpha_{1}, \alpha_{2}, ... , \alpha_{2g + 1}\}$.  The \textit{pure} braid group on $2g + 1$ strands, denoted $P_{2g + 1}$, is defined to be the kernel of the surjective homomorphism from $B_{2g + 1}$ to the symmetric group $S_{2g + 1}$ which sends $\sigma_{i}$ to $(i, i + 1) \in S_{2g + 1}$ for $1 \leq i \leq 2g$ (see the proof of Theorem 1.8 in \cite{birman1974braids}).  Then $P_{2g + 1} \lhd B_{2g + 1}$ is the (normal) subgroup corresponding to the cover $Y \to X$, and is therefore isomorphic to the fundamental group of $Y$.

Let $\mathcal{O}_{X}$ denote the coordinate ring of $X$, and let $F(x) \in \mathcal{O}_{X}[x]$ be the degree-$(2g + 1)$ polynomial given by 
\begin{equation}\label{universal polynomial} x^{2g + 1} + \sum_{i = 1}^{2g + 1} (-1)^{i}e_{i}x^{2g + 1 - i}.\end{equation}
Now denote by $\mathcal{C} \rightarrow X$ the affine scheme defined by the equation $y^{2} = F(x)$.  Clearly, $\mathcal{C}$ is the family over $X$ whose fiber over a point $T \in X(\cc)$ is the smooth affine hyperelliptic curve defined by $y^{2} = \prod_{z \in T} (x - z)$, and the generic fiber of $\mathcal{C}$ is $C / L$.  Fix a basepoint $T_{0}$ of $X$, and a basepoint $P_{0}$ of $\mathcal{C}_{T_{0}}$.  Then we have a short exact sequence of fundamental groups 
\begin{equation}\label{SES of fundamental groups} 1 \rightarrow \pi_{1}(\mathcal{C}_{T_{0}}, P_{0}) \rightarrow \pi_{1}(\mathcal{C}, P_{0}) \rightarrow \pi_{1}(X, T_{0}) \rightarrow 1. \end{equation}

We now construct a continuous section $s: X \rightarrow \mathcal{C}$, following the proof of Lemma 6.1 and the discussion in \cite{Yu1997toward}, \S6.  For $i = 1, 2$, let $\mathcal{E}_{i} \to X$ be the affine scheme given by $\mathrm{Spec}(\mathcal{O}_{X}[x, y]/(y^{i} - F(x))[F(x)^{-1}])$.  Then $\mathcal{E}_{1} \to X$ is clearly the family of complex topological spaces whose fiber over a point $T \in X$ can be identified with $\cc \setminus T$, and there is an obvious degree-$2$ cover $\mathcal{E}_{2} \to \mathcal{E}_{1}$.  Let $t : X \to \mathcal{E}_{1}$ be the continuous map of complex topological spaces which sends a point $T \in X$ to $\mathrm{max}_{z \in T}\{|z|\} + 1 \in \cc \setminus T = \mathcal{E}_{1, T}$.  This section then lifts to a section $\tilde{t} : X \to \mathcal{E}_{2}$.  Define $s : X \rightarrow \mathcal{C}$ to be the composition of $\tilde{t}$ with the obvious inclusion map $\mathcal{E}_{2} \hookrightarrow \mathcal{C}$.  It is easy to check from the construction of $s$ that it is a section of the family $\mathcal{C} \to X$.

The section $s$ induces a monodromy action of $\pi_{1}(X, T_{0})$ on $\pi_{1}(\mathcal{C}_{T_{0}}, P_{0})$, which is given by $\sigma \in \pi_{1}(X)$ acting as conjugation by $s(\sigma)$ on $\pi_{1}(\mathcal{C}_{T_{0}}, P_{0}) \lhd \pi_{1}(\mathcal{C}, P_{0})$.  This induces an action of $B_{2g + 1}$ on the abelianization of $\pi_{1}(\mathcal{C}_{T_{0}}, P_{0})$, the homology group $H_{1}(\mathcal{C}_{T_{0}}, \zz)$, which is isomorphic to $\zz^{2g}$.  We denote this action by
\begin{equation}\label{braid group representation} R: B_{2g + 1} \cong \pi_{1}(X, T_{0}) \rightarrow \mathrm{Aut}(H_{1}(\mathcal{C}_{T_{0}}, \zz)). \end{equation}
This action respects the intersection pairing on $\mathcal{C}_{T_{0}}$, so the image of $R$ is actually contained in the corresponding subgroup of symplectic automorphisms $\mathrm{Sp}(H_{1}(\mathcal{C}_{T_{0}}, \zz))$.

The following theorem is proven in \cite{a1979tresses} (Th\'{e}or\'{e}me 1), as well as in \cite{Mumford1984tata} (Lemma 8.12).

\begin{thm}\label{thm: from Mumford}
In the representation $R: B_{2g + 1} \rightarrow \mathrm{Sp}(H_{1}(\mathcal{C}_{T_{0}}, \zz))$, the image of $P_{2g + 1}$ coincides with $\Gamma(2)$.
\end{thm}

Let $\widehat{B}_{2g + 1}$ denote the profinite completion of $B_{2g + 1} \cong \pi_{1}(X, T_{0})$.  Since $X$ may be viewed as a scheme over the complex numbers, Riemann's Existence Theorem yields an isomorphism between its \'{e}tale fundamental group $\pi_{1}^{\acute{e}t}(X, T_{0})$ and $\widehat{B}_{2g + 1}$  (\cite{grothendieck2002rev}, Expos\'{e} XII, Corollaire 5.2).  Meanwhile, $\pi_{1}^{\acute{e}t}(X, T_{0})$ is isomorphic to the Galois group $\Gal(L^{\mathrm{unr}} / L)$, where $L^{\mathrm{unr}}$ is the maximal extension of $L$ unramified at all points of $X$.  The representation $R : B_{2g + 1} \to \mathrm{Sp}(H_{1}(\mathcal{C}_{T_{0}}, \zz))$ induces a homomorphism of profinite groups 
\begin{equation} R : \Gal(L^{\mathrm{unr}} / L) = \widehat{B}_{2g + 1} \to \mathrm{Sp}(H_{1}(\mathcal{C}_{T_{0}}, \zz) \otimes \zz_{\ell}) \end{equation}
for any prime $\ell$.  Composing this map with the restriction homomorphism $G_{L} := \Gal(\bar{L} / L) \twoheadrightarrow \Gal(L^{\mathrm{unr}} / L)$ yields a map which we denote $R_{\ell} : G_{L} \to \mathrm{Sp}(H_{1}(\mathcal{C}_{T_{0}}, \zz) \otimes \zz_{\ell})$.  The following proposition will allow us to convert the above topological result into the arithmetic statement of Theorem \ref{thm: section 2 main theorem}.

\begin{prop}\label{prop: isomorphic representations}

Assume the above notation, and let $\ell$ be any prime.  Then there is an isomorphism of $\zz_{\ell}$-modules $T_{\ell}(J) \stackrel{\sim}{\to} H_{1}(\mathcal{C}_{T_{0}}, \zz) \otimes \zz_{\ell}$ making the representations $\rho_{\ell}$ and $R_{\ell}$ isomorphic.

\end{prop}

\begin{proof}

We proceed in five steps.

\emph{Step 1:} We switch from the affine curve $C$ to a smooth compactification of $C$, which is defined as follows.  Let $C'$ be the (smooth) curve defined over $L$ by the equation 
\begin{equation} \label{other affine piece} y'^{2} = x' \prod_{i = 1}^{2g + 1} (1 - \alpha_{i}x'). \end{equation}
We glue the open subset of $C$ defined by $x \neq 0$ to the open subset of $C'$ defined by $x' \neq 0$ via the mapping 
$$x' \mapsto \frac{1}{x},\ y' \mapsto \frac{y}{x^{g + 1}}, $$
and denote the resulting smooth, projective scheme by $\bar{C}$.  (See \cite{Mumford1984tata}, \S1 for more details of this construction.)  Let $\infty \in \bar{C}(L)$ denote the ``point at infinity" given by $(x', y') = (0, 0) \in C'$.  The curve $\bar{C}$ has smooth reduction over every point $T \in X$ and therefore can be extended in an obvious way to a family $\bar{\mathcal{C}} \to X$ whose generic fiber is $\bar{C} / L$.  Note that $\bar{\mathcal{C}}_{T}$ is a smooth compactification of $\mathcal{C}_{T}$ for each $T \in X$.  There is a surjective map $\pi_{1}(\mathcal{C}_{T_{0}}, P_{0}) \twoheadrightarrow \pi_{1}(\bar{\mathcal{C}}_{T_{0}}, \infty_{T_{0}})$ induced by the inclusion $\mathcal{C} \hookrightarrow \bar{\mathcal{C}}$.  Note also that the section $s : X \to \mathcal{C} \subset \bar{\mathcal{C}}$ can be continuously deformed to the ``constant section" $\bar{s} : X \to \bar{\mathcal{C}}$ sending each $T \in X$ to the point at infinity $\infty_{T} \in \mathcal{C}_{T}$.  Therefore, $\bar{s}_{*} : \pi_{1}(X, T_{0}) \to \pi_{1}(\bar{\mathcal{C}}_{T_{0}}, \infty_{T_{0}})$ is the composition of $s_{*}$ with the map $\pi_{1}(\mathcal{C}_{T}) \twoheadrightarrow \pi_{1}(\bar{\mathcal{C}}_{T})$.  In this way, we may view the action of $\pi_{1}(X, T_{0})$ on $\pi_{1}(\mathcal{C}_{T_{0}}, P_{0})^{\mathrm{ab}} = \pi_{1}(\bar{\mathcal{C}}_{T_{0}}, \infty_{T_{0}})^{\mathrm{ab}}$ as being induced by $\bar{s}_{*}$.

\emph{Step 2:} We switch from (topological) fundamental groups to \'{e}tale fundamental groups.  Since $X$ and $\mathcal{C}$, as well as $\mathcal{C}_{T}$ for each $T \in X$, can be viewed as a scheme over the complex numbers, Riemann's Existence Theorem implies that the \'{e}tale fundamental groups of $X$, $\mathcal{C}$, and each $\mathcal{C}_{T}$ (defined using a choice of geometric base point $\bar{T_{0}}$ over $T_{0}$) are isomorphic to the profinite completions of their respective topological fundamental groups.  Taking profinite completions induces a sequence of \'{e}tale fundamental groups
\begin{equation}\label{SES of fundamental groups of Jacobians} 1 \to \pi_{1}^{\acute{e}t}(\mathcal{C}_{\bar{T_{0}}}, 0_{\bar{T_{0}}}) \to \pi_{1}^{\acute{e}t}(\mathcal{C}, 0_{\bar{T_{0}}}) \to \pi_{1}^{\acute{e}t}(X, \bar{T_{0}}) \to 1, \end{equation}
which is a short exact sequence by \cite{grothendieck2002rev}, Corollaire X.2.2.  Moreover, the section $\bar{s}: X \to \bar{\mathcal{C}}$ similarly gives rise to an action of $\pi_{1}^{\acute{e}t}(X, \bar{T_{0}})$ on $\pi_{1}^{\acute{e}t}(\bar{\mathcal{C}}_{T_{0}}, \infty_{\bar{T_{0}}})^{\mathrm{ab}}$.

\emph{Step 3:} We switch from $\bar{C}$ to its Jacobian.  Define $\mathcal{J} \to X$ to be the abelian scheme representing the Picard functor of the scheme $\mathcal{C} \to X$ (see \cite{milne1986jacobian}, Theorem 8.1).  Note that $\mathcal{J}_{T}$ is the Jacobian of $\mathcal{C}_{T}$ for each $\cc$-point $T$ of $X$, and the generic fiber of $\mathcal{J}$ is $J / L$, the Jacobian of $C / L$.  Let $f_{\infty} : \bar{C} \to J$ be the morphism (defined over $L$) given by sending each point $P \in \bar{C}(L)$ to the divisor class $[(P) - (\infty)]$ in $\mathrm{Pic}_{L}^{0}(\bar{C})$, which is identified with $J(L)$.  By \cite{milne1986jacobian}, Proposition 9.1, the induced homomorphism of \'{e}tale fundamental groups $(f_{\infty})_{*} : \pi_{1}^{\acute{e}t}(\bar{C}, \infty) \to \pi_{1}^{\acute{e}t}(J, 0)$ factors through an isomorphism $\pi_{1}^{\acute{e}t}(\bar{C}, \infty)^{\mathrm{ab}} \stackrel{\sim}{\to} \pi_{1}^{\acute{e}t}(J, 0)$.  This induces an isomorphism $\pi_{1}^{\mathrm{\acute{e}t}}(\bar\mathcal{{C}}_{T}, \infty_{T})^{\mathrm{ab}} \stackrel{\sim}{\to} \pi_{1}^{\mathrm{\acute{e}t}}(\mathcal{J}_{T}, 0_{T})$ for each $T \in X$.  Note that the composition of the section $\bar{s} : X \to \bar{\mathcal{C}}$ with $f_{\infty}$ is the ``zero section" $o : X \to \mathcal{J}$ mapping each $T$ to the identity element $0_{T} \in \mathcal{J}_{T}$.  Thus, the action of $\pi_{1}^{\acute{e}t}(X, \bar{T_{0}})$ on $\pi_{1}^{\acute{e}t}(\mathcal{C}_{T_{0}}, \infty_{\bar{T_{0}}})^{\mathrm{ab}}$ coming from the splitting of (\ref{SES of fundamental groups}) is the same as the action of $\pi_{1}^{\acute{e}t}(X, \bar{T_{0}})$ on $\pi_{1}^{\acute{e}t}(\mathcal{J}_{\bar{T_{0}}}, 0_{\bar{T_{0}}})$ coming from the splitting of (\ref{SES of fundamental groups of Jacobians}) induced by the section $o_{*} : \pi_{1}^{\acute{e}t}(X, \bar{T_{0}}) \to \pi_{1}^{\acute{e}t}(\mathcal{J}, 0_{\bar{T_{0}}})$.

\emph{Step 4:} We now show that this action on $\pi_{1}^{\acute{e}t}(\mathcal{J}_{\bar{T_{0}}}, 0_{\bar{T_{0}}})$ is isomorphic to a Galois action on $\pi_{1}^{\acute{e}t}(J_{\bar{L}}, 0)$ (and therefore on its $\ell$-adic quotient $T_{\ell}(J)$).  Let $\eta : \mathrm{Spec}(L) \to X$ denote the generic point of $X$.  Note that we may identify $\pi_{1}^{\acute{e}t}(L, \bar{L})$ with $G_{L}$, and that $\eta_{*} : G_{L} \twoheadrightarrow \pi_{1}^{\acute{e}t}(X, \bar{\eta})$ is a surjection (in fact, it is the restriction homomorphism of Galois groups corresponding to the maximal algebraic extension of $L$ unramified at all points of $X$).  Also, the point $0 \in J_{L}$ may be viewed as a morphism $0 : \mathrm{Spec}(L) \to J_{L}$ which induces $0_{*} : G_{L} = \pi_{1}^{\acute{e}t}(L, \bar{L}) \to \pi_{1}^{\acute{e}t}(J_{L}, 0)$.  Let $\bar{T_{0}}$ and $\bar{\eta}$ be geometric points over $T_{0}$ and $\eta$ respectively.  Then we have (\cite{grothendieck2002rev}, Corollaire X.1.4) an exact sequence of \'{e}tale fundamental groups
\begin{equation} \label{SES of fundamental groups generic} \pi_{1}^{\acute{e}t}(\mathcal{J}, 0_{\bar{\eta}}) \to \pi_{1}^{\acute{e}t}(\mathcal{J}, 0_{\bar{\eta}}) \to \pi_{1}^{\acute{e}t}(X, \bar{\eta}) \to 1. \end{equation}
Changing the geometric basepoint of $X$ from $\bar{T_{0}}$ to $\bar{\eta}$ (resp. changing the geometric basepoint of $\mathcal{J}$ from $0_{\bar{\eta}}$ to $0_{\bar{T_{0}}}$) non-canonically induces an isomorphism $\pi_{1}^{\acute{e}t}(X, \bar{\eta}) \stackrel{\sim}{\to} \pi_{1}^{\acute{e}t}(X, \bar{T_{0}})$ (resp. an isomorphism $\pi_{1}^{\acute{e}t}(\mathcal{J}, 0_{\bar{\eta}}) \stackrel{\sim}{\to} \pi_{1}^{\acute{e}t}(\mathcal{J}, 0_{\bar{T_{0}}})$).  Fix such an isomorphism $\varphi: \pi_{1}^{\acute{e}t}(X, \bar{\eta}) \stackrel{\sim}{\to} \pi_{1}^{\acute{e}t}(X, \bar{T_{0}})$.  Then we have the following commutative diagram, where all horizontal rows are exact:
$$ \xymatrix{ 1 \ar[r]  & \pi_{1}^{\acute{e}t}(J_{\bar{L}}, 0) \ar[r] \ar@{=}[d] & \pi_{1}^{\acute{e}t}(J_{L}, 0) \ar[r] \ar@{->>}[d] & \pi_{1}^{\acute{e}t}(L, \bar{L}) \ar@/_1pc/[l]_{0_{*}} \ar[r] \ar@{->>}[d]^{\eta_{*}} & 1
\\    & \pi_{1}^{\acute{e}t}(\mathcal{J}_{\bar{\eta}}, 0_{\bar{\eta}}) \ar[r] \ar@{-->}[d]^{\mathrm{sp}} & \pi_{1}^{\acute{e}t}(\mathcal{J}, 0_{\bar{\eta}}) \ar[r] \ar[d]^{\wr} & \pi_{1}^{\acute{e}t}(X, \bar{\eta}) \ar@/_1pc/[l]_{o_{*}} \ar[r] \ar[d]^{\wr \; \varphi} & 1 
\\ 1 \ar[r] & \pi_{1}^{\acute{e}t}(\mathcal{J}_{\bar{T_{0}}}, 0_{\bar{T_{0}}}) \ar[r] & \pi_{1}^{\acute{e}t}(\mathcal{J}, 0_{\bar{T_{0}}}) \ar[r] & \pi_{1}^{\acute{e}t}(X, \bar{T_{0}}) \ar@/_1pc/[l]_{o_{*}} \ar[r] & 1 }$$
Here the vertical arrow from $\pi_{1}^{\acute{e}t}(\mathcal{J}, 0_{\bar{\eta}})$ to $\pi_{1}^{\acute{e}t}(\mathcal{J}, 0_{\bar{T_{0}}})$ is a change-of-basepoint isomorphism chosen to make the lower right square commute, and $\mathrm{sp} : \pi_{1}^{\acute{e}t}(\mathcal{J}_{\bar{\eta}}, 0_{\bar{\eta}}) \to \pi_{1}^{\acute{e}t}(\mathcal{J}_{\bar{T_{0}}}, 0_{\bar{T_{0}}})$ is the surjective homomorphism induced by a diagram chase on the bottom two horizontal rows.  Grothendieck's Specialization Theorem (\cite{grothendieck2002rev}, Corollaire X.3.9) states that $\mathrm{sp}$ is an isomorphism, which implies that the second row is also a short exact sequence.  Thus, the action of $\pi_{1}^{\acute{e}t}(X, \bar{T_{0}})$ on $\pi_{1}^{\acute{e}t}(\mathcal{J}_{\bar{T_{0}}}, 0_{\bar{T_{0}}})$ arising from the splitting of the lower row by $o_{*}$ is isomorphic to the action of $\pi_{1}^{\acute{e}t}(X, \bar{\eta})$ on $\pi_{1}^{\acute{e}t}(\mathcal{J}_{\bar{\eta}}, 0_{\bar{\eta}})$ arising from the splitting of the middle row by $o_{*}$, via the isomorphism $\mathrm{sp} : \pi_{1}^{\acute{e}t}(\mathcal{J}_{\bar{\eta}}, 0_{\bar{\eta}}) \to \pi_{1}^{\acute{e}t}(\mathcal{J}_{\bar{T_{0}}}, 0_{\bar{T_{0}}})$.  In turn, a simple diagram chase confirms that this action, after pre-composing with $\eta_{*} : \pi_{1}^{\acute{e}t}(L, \bar{L}) \twoheadrightarrow \pi_{1}^{\acute{e}t}(X, \bar{\eta})$, can be identified with the action of $\pi_{1}^{\acute{e}t}(L, \bar{L})$ on $\pi_{1}^{\acute{e}t}(J_{\bar{L}}, 0)$ arising from the splitting of the top row by $0_{*}$.  We denote this action by $\tilde{R} : G_{L} = \pi_{1}^{\acute{e}t}(L, \bar{L}) \to \mathrm{Aut}(\pi_{1}^{\acute{e}t}(J_{\bar{L}}, 0))$.  Since the Tate module $T_{\ell}(J)$ may be identified with the maximal pro-$\ell$ quotient of $\pi_{1}^{\acute{e}t}(J_{\bar{L}}, 0)$, $\tilde{R}$ induces an action of $G_{L}$ on $T_{\ell}(J)$, which we denote by $\tilde{R}_{\ell} : G_{L} \to \mathrm{Aut}(T_{\ell}(J))$.  One can identify the symplectic pairing on $\pi_{1}(\mathcal{J}_{T_{0}}, 0_{T_{0}})$ with the Weil pairing on $T_{\ell}(J)$ via the results in \cite{mumford1974abelian}, Chapter IV, \textsection 24.  Therefore, the image of $\tilde{R}_{\ell}$ is a subgroup of $\mathrm{Sp}(T_{\ell}(J))$.

By the above construction, we may identify the maximal pro-$\ell$ quotient of $\pi_{1}^{\acute{e}t}(\mathcal{J}_{\bar{T_{0}}}, 0_{\bar{T_{0}}})$ with $H_{1}(\mathcal{C}_{T_{0}}, \zz) \otimes \zz_{\ell}$.  Note that the isomorphism $\mathrm{sp} : \pi_{1}^{\acute{e}t}(\mathcal{J}_{\bar{\eta}}, 0_{\bar{\eta}}) \stackrel{\sim}{\to} \pi_{1}^{\acute{e}t}(\mathcal{J}_{\bar{T_{0}}}, 0_{\bar{T_{0}}})$ induces an isomorphism of their maximal pro-$\ell$ quotients $\mathrm{sp}_{\ell} : T_{\ell}(J) \stackrel{\sim}{\to} H_{1}(\mathcal{C}_{T_{0}}, \zz) \otimes \zz_{\ell}$.  By construction, the representation $\tilde{R}_{\ell}$ is isomorphic to the representation $R_{\ell}$ via $\mathrm{sp}_{\ell}$.

\emph{Step 5:} It now suffices to show that $\tilde{R}_{\ell} = \rho_{\ell}$.  To determine $\tilde{R}_{\ell}$, we are interested in the action of $G_{L}$ on the group $\mathrm{Aut}_{J_{\bar{L}}}(Z)$ for each $\ell$-power-degree covering $Z \to J_{\bar{L}}$.  But each such covering is a subcovering of $[\ell^{n}] : J_{\bar{L}} \to J_{\bar{L}}$, so it suffices to determine the action of $G_{L}$ on the group of translations $\{t_{P} | P \in J[\ell^{n}]\}$ for each $n$.  Recall that $0_{*} : G_{L} \to \pi_{1}^{\acute{e}t}(J_{L}, 0)$ is induced by the inclusion of the $L$-point $0 \in J_{L}$.  Thus, for any $\sigma \in G_{L}$, $0_{*}(\sigma)$ acts on any connected \'{e}tale cover of $J_{L}$ via $\sigma$ acting on the coordinates of the points.  Since $\tilde{R}(\sigma)$ is conjugation by $0_{*}(\sigma)$ on $\pi_{1}^{\acute{e}t}(J_{\bar{L}}, 0) \lhd \pi_{1}^{\acute{e}t}(J_{L}, 0)$, one sees that for each $n$, $0_{*}(\sigma)$ acts on $\{t_{P} | P \in J[\ell^{n}]\}$ by sending each $t_{P}$ to $\sigma^{-1}t_{P}\sigma = t_{P^{\sigma}}$.  Thus, $G_{L}$ acts on the Galois group of the covering $[\ell^{n}] : J_{\bar{L}} \to J_{\bar{L}}$ via the usual Galois action on $J[\ell^{n}]$.  This lifts to the usual action of $G_{L}$ on $T_{\ell}(J)$, and we are done.

\end{proof}

It is now easy to prove the main theorem.

\theoremstyle{remark}
\newtheorem*{proof of section 2 main theorem}{Proof (of Theorem \ref{thm: section 2 main theorem}) }
\begin{proof of section 2 main theorem}

Recall that $P_{2g + 1}$ is the normal subgroup of $B_{2g + 1} \cong \pi_{1}(X, T_{0})$ corresponding to the cover $Y \to X$, and the function field of $Y$ is $\cc(\alpha_{1}, ... , \alpha_{2g + 1}) = L_{1}$.  It follows that the image of $\mathrm{Gal}(\bar{L} / L_{1})$ under $\eta_{*}$ is $\widehat{P}_{2g + 1} \lhd \widehat{B}_{2g + 1} \cong \pi_{1}^{\acute{e}t}(X, \bar{T_{0}})$ (where $\widehat{P}_{2g + 1}$ denotes the profinite completion of $P_{2g + 1}$).  Therefore, the statement of Theorem \ref{thm: from Mumford} with $\ell = 2$ implies that the image of $\mathrm{Gal}(\bar{L} / L_{1})$ under $R_{2}$ is $\Gamma(2) \lhd \mathrm{Sp}(H_{1}(\mathcal{C}_{T_{0}}, \zz) \otimes \zz_{\ell})$.  It then follows from the statement of Lemma \ref{prop: isomorphic representations} that the image of $\mathrm{Gal}(\bar{L} / L_{1})$ under $\rho_{2}$ is $\Gamma(2) \lhd \mathrm{Sp}(T_{2}(J))$.

\qed
\end{proof of section 2 main theorem}

\section{Fields of $4$-torsion}

One application of Theorem \ref{thm: section 2 main theorem} is that it allows us to obtain an explicit description of $L_{2}$.  We will follow Yu's argument in \cite{Yu1997toward}.

\begin{prop}\label{prop: 2 and 4 case}
We have

$$L_{2} = L_{1}(\{\sqrt{\alpha_{i} - \alpha_{j}}\}_{1 \leq i < j \leq 2g + 1}).$$

\end{prop}

\begin{proof}

For $n \geq 1$, let $\mathcal{B}_{n}$ denote the set of bases of the free $\zz /2^{n}\zz$-module $\mathcal{J}_{T_{0}}[2^{n}]$.  Then it was shown in the proof of Theorem \ref{thm: section 2 main theorem} that $G_{L}$ acts on $\mathcal{B}_{n}$ through the map $R : \pi_{1}(X, T_{0}) \to \mathrm{Sp}(H_{1}(\mathcal{C}_{T_{0}}, \zz)) = \mathrm{Sp}(H_{1}(\mathcal{J}_{T_{0}}, \zz))$ in the statement of Theorem \ref{thm: from Mumford}, and the subgroup fixing all elements of $\mathcal{B}_{n}$ corresponds to $R^{-1}(\Gamma(2^{n})) \lhd \pi_{1}(X, T_{0})$.  Hence, by covering space theory, there is a connected cover $X_{n} \to X$ corresponding to an orbit of $\mathcal{B}_{n}$ under the action of $\pi_{1}(X, T_{0})$, and the function field of $X_{n}$ is the extension of $L$ fixed by the subgroup of $G_{L}$ which fixes all bases of $J[2^{n}]$.  Clearly, this extension is $L_{n}$.  Thus, the Galois cover $X_{n} \to X$ is an unramified morphism of connected affine schemes corresponding to the inclusion $L \hookrightarrow L_{n}$ of function fields.

Note that, setting $n = 1$, we get that $X_{1}$ is the Galois cover of $X$ whose \'{e}tale fundamental group can be identified with $R^{-1}(\Gamma(2)) \lhd \pi_{1}(X, T_{0})$.  Theorem \ref{thm: from Mumford} implies that $R^{-1}(\Gamma(2))$ is isomorphic to $\widehat{P}_{2g + 1}$, the profinite completion of $P_{2g + 1}$.  For $n \geq 1$, the \'{e}tale morphism $X_{n} \to X_{1}$ corresponds to the function field extension $L_{n} \supset L_{1}$, which by Corollary \ref{prop: description of G}(c) has Galois group isomorphic to $\Gamma(2) / \Gamma(2^{n})$.  Therefore, $X_{n}$ is the cover of $X_{1}$ whose \'{e}tale fundamental group can be identified with a normal subgroup of $\widehat{P}_{2g + 1}$ with quotient isomorphic to $\Gamma(2) / \Gamma(2^{n})$.

In the proof of Corollary 2.2 of \cite{sato2010abelianization}, it is shown that $\Gamma(2) / \Gamma(4) \cong (\zz / 2\zz)^{2g^{2} + g}$, and thus, 
\begin{equation}\label{Gamma(2)/Gamma(4)} \mathrm{Gal}(L_{2} / L_{1}) \cong \Gamma(2) / \Gamma(4) \cong (\zz / 2\zz)^{2g^{2} + g}. \end{equation}
It is also clear from looking at a presentation of the pure braid group $P_{2g + 1}$ (see for instance \cite{birman1974braids}, Lemma 1.8.2) that the abelianization of $P_{2g + 1}$ is a free abelian group of rank $2g^{2} + g$.  Therefore, its maximal abelian quotient of exponent $2$ is isomorphic to $(\zz / 2\zz)^{2g^{2} + g}$.  Thus, $\widehat{P}_{2g + 1}$ has a unique normal subgroup inducing a quotient isomorphic to $(\zz / 2\zz)^{2g^{2} + g}$.  It follows that there is only one Galois cover of $X_{1}$ with Galois group isomorphic to $\Gamma(2) / \Gamma(4)$, namely $X_{2}$.  The field extension $L_{1}(\{\sqrt{\alpha_{i} - \alpha_{j}}\}_{i < j}) \supset L_{1}$ is unramified away from the hyperplanes defined by $(\alpha_{i} - \alpha_{j})$ with $i \neq j$ and is obtained from $L_{1}$ by adjoining $2g^{2} + g$ independent square roots of elements in $L_{1}^{\times} \backslash (L_{1}^{\times})^{2}$.  Therefore, $L_{1}(\{\sqrt{\alpha_{i} - \alpha_{j}}\}_{i < j})$ is the function field of a Galois cover of $X(2)$ with Galois group isomorphic to $(\zz / 2\zz)^{2g^{2} + g} \cong \Gamma(2) / \Gamma(4)$.  It follows that this cover of $X_{1}$ is $X_{2}$, and that $L_{1}(\{\sqrt{\alpha_{i} - \alpha_{j}}\}_{i < j})$ is $L_{2}$, the function field of $X_{2}$.

\end{proof}

\section{Generalizations}

As in Section 1, let $k$ be an algebraic extension of $\qq$ which contains all $2$-power roots of unity, and let $K$ be the transcendental extension obtained by adjoining the coefficients of (\ref{hyperelliptic model}) to $k$.  We will also fix the following notation.  Let $C_{K}$ be the hyperelliptic curve defined over $K$ given by the equation (\ref{hyperelliptic model}), and let $J_{K}$ be its Jacobian.  For each $n \geq 0$, let $K_{n}$ be the extension of $K$ over which the $2^{n}$-torsion of $J_{K}$ is defined.  Note that, analogous to the situation with $C / L$, the extension $K_{2}$ is $k(\alpha_{1}, ... , \alpha_{2g + 1})$, which is Galois over $K$ with Galois group isomorphic to $S_{2g + 1}$.  Let $\rho_{2, K} : \mathrm{Gal}(K_{\infty} / K) \to \mathrm{Sp}(T_{2}(J_{K}))$ be the homomorphism arising from the Galois action on the Tate module of $J_{K}$.  We now investigate what happens to the Galois action when we descend from working over $\cc$ to working over $k$.  (In what follows, we canonically identify $T_{2}(J)$ with $T_{2}(J_{K})$ and $\Gamma(2^{n})$ with the level-$2^{n}$ congruence subgroup of $\mathrm{Sp}(T_{2}(J_{K}))$ for each $n \geq 0$.)

\begin{prop}\label{prop: descent}

The statements of Theorem \ref{thm: section 2 main theorem}, Corollary \ref{prop: description of G}, and Proposition \ref{prop: 2 and 4 case} are true when $L$ and $\rho_{2}$ are replaced by $K$ and $\rho_{2, K}$ respectively.

\end{prop}

\begin{proof}

For any $n \geq 0$, let $\theta_{n} : \mathrm{Gal}(L_{\infty} / L_{n}) \to \mathrm{Gal}(K_{\infty} / K_{n})$ be the composition of the obvious inclusion $\mathrm{Gal}(L_{\infty} / L_{n}) \hookrightarrow \mathrm{Gal}(L_{\infty} / K_{n})$ with the obvious restriction map $\mathrm{Gal}(L_{\infty} / K_{n}) \twoheadrightarrow \mathrm{Gal}(K_{\infty} / K_{n})$.  Let $\bar{\rho}_{2}^{(\infty)}$ (resp. $\bar{\rho}_{2, K}^{(\infty)}$) be the representation of $\mathrm{Gal}(L_{\infty} / L)$ (resp. $\mathrm{Gal}(K_{\infty} / K)$) induced from $\rho_{2}$ (resp. $\rho_{2, K}$) by the restriction homomorphism of the Galois groups.  It is easy to check that $\bar{\rho}_{2}^{(\infty)} = \bar{\rho}_{2, K}^{(\infty)} \circ \theta_{0}$.  It will suffice to show that $\theta_{0}$ is an isomorphism.

First, note that for any $n \geq 0$, $\theta_{n}$ is injective by the linear disjointness of $K_{\infty}$ and $L_{n}$ over $K_{n}$.  Now suppose that $n \geq 1$.  Then, as in the proof of Corollary \ref{prop: description of G}, the image under $\bar{\rho}$ of $\mathrm{Gal}(L_{\infty} / L_{n})$ is the entire congruence subgroup $\Gamma(2^{n})$.  Therefore, since $\theta_{n}$ is injective, the image under $\bar{\rho}_{K}$ of $\mathrm{Gal}(K_{\infty} / K_{n})$ contains $\Gamma(2^{n})$.  But since $K$ contains all $2$-power roots of unity, the Weil pairing is Galois invariant, and so the image of $\mathrm{Gal}(K_{\infty} / K_{n})$ must also be contained in $\Gamma(2^{n})$.  Therefore, $\theta_{n}$ is an isomorphism for $n \geq 1$.  Now, using Corollary \ref{prop: description of G}(a) and the fact that $\mathrm{Gal}(K(\alpha_{1}, ... , \alpha_{2g + 1}) / K) \cong S_{2g + 1}$, we get the commutative diagram below, whose top and bottom rows are short exact sequences.

$$ \xymatrix{ 1 \ar[r]  & \mathrm{Gal}(L_{\infty} / L_{1}) \ar[r] \ar[d]^{\theta_{1}} & \mathrm{Gal}(L_{\infty} / L) \ar[r] \ar[d]^{\theta_{0}} & S_{2g + 1} \ar[r] \ar@{=}[d]& 1
\\ 1 \ar[r] & \mathrm{Gal}(K_{\infty} / K_{1}) \ar[r] & \mathrm{Gal}(K_{\infty} / K) \ar[r] & S_{2g + 1} \ar[r] & 1 }$$

By the Short Five Lemma, since $\theta_{1}$ is an isomorphism, so is $\theta_{0}$.

\end{proof}

\begin{rmk}\label{rmk: generalizations}

a) Suppose we drop the assumption that $k$ contains all $2$-power roots of unity.  Then $\rho_{2, K}(G_{K})$ is no longer contained in $\mathrm{Sp}(T_{2}(J))$ in general.  However, the Galois equivariance of the Weil pairing forces the image of $\rho_{2, K}$ to be contained in the group of symplectic similitudes 
$$\mathrm{GSp}(T_{2}(J)) := \{\sigma \in \mathrm{Aut}(T_{2}(J))\ |\ E_{2}(P^{\sigma}, Q^{\sigma}) = \chi_{2}(\sigma)E_{2}(P, Q)\ \forall P, Q \in T_{2}(J)\},$$
where $\displaystyle E_{2} : T_{2}(J) \times T_{2}(J) \to \lim_{\leftarrow n} \mu_{2^{n}} \cong \zz_{2}$ is the Weil pairing on the $2$-adic Tate module of $J$, and $\chi_{2} : G_{K} \to \zz_{2}^{\times}$ is the cyclotomic character on the absolute Galois group of $K$.  Galois equivariance of the Weil pairing also implies that $K_{\infty}$ contains all $2$-power roots of unity.  Thus, $K_{\infty} \supset K(\mu_{2^{\infty}})$, and the statements referred to in Proposition \ref{prop: descent} still hold when we replace $K$ with $K(\mu_{2^{\infty}})$.

Furthermore, if $K$ contains $\sqrt{-1}$, the Weil pairing on $J[4]$ is Galois invariant, so the image of $\Gal(K_{2} / K_{1})$ coincides with $\Gamma(2) / \Gamma(4) \lhd \mathrm{Sp}(J[4])$ and is therefore isomorphic to $\Gal(L_{2} / L_{1})$.  It follows that Proposition \ref{prop: 2 and 4 case} still holds over $K(\sqrt{-1})$; that is, 
\begin{equation}\label{4-torsion generalization} K_{2} = K_{1}(\sqrt{-1}, \{\sqrt{\alpha_{i} - \alpha_{j}}\}_{1 \leq i < j \leq 2g + 1}). \end{equation}

b) In addition, suppose that $k$ is finitely generated over $\qq$ (for example, a number field).  We may specialize by assigning an element of $k$ to each coefficient of the degree-$(2g + 1)$ polynomial in (\ref{hyperelliptic model}), and defining the corresponding Jacobian $J_{k} / k$ and Galois representation $\rho_{2, k} : G_{k} \to \mathrm{Sp}(T_{2}(J_{k}))$.   Then we may use Proposition 1.3 of \cite{noot1995abelian} and its proof (see also \cite{serre4lettres}) to see that for infinitely many choices of $e_{1}, ... , e_{2g + 1} \in k$, $\rho_{2, k}(G_{k})$ can be identified with $\rho_{2, K}(G_{K})$ from part (a).  We have $\rho_{2, k}(\mathrm{Gal}(\bar{k} / k(\mu_{2^{\infty}}))) = \rho_{2, k}(G_{k}) \cap \mathrm{\mathrm{Sp}}(T_{2}(J_{k}))$, and therefore, the statements referred to in Proposition \ref{prop: descent} still hold over $k(\mu_{2^{\infty}})$.  Similarly, Proposition \ref{prop: 2 and 4 case} still holds over $k(\sqrt{-1})$.

\end{rmk}

\section*{Acknowledgements}

I am grateful to Yuri Zarhin for his many ideas and suggestions.  I would also like to thank the referee, who suggested that this material be presented in a separate paper, and whose corrections were helpful in improving the exposition.

\bibliographystyle{plain}
\bibliography{bibfile}

\end{document}